\documentclass[12pt,a4paper]{amsart}

\usepackage{amsmath,amssymb}
\usepackage{amsthm}
\usepackage{mathrsfs}

\usepackage{geometry}
\newcommand\R{\mathbb{R}}

\newcommand\cB{\mathcal{B}}
\newcommand\cC{\mathcal{C}}

\newtheorem{theorem}{Theorem}

\newtheorem{lemma}[theorem]{Lemma}

\newtheorem{definition}[theorem]{Definition}

\theoremstyle{remark}
\newtheorem{rem}[theorem]{Remark}

\numberwithin{equation}{section}
\numberwithin{theorem}{section}

\newcommand{\astq}{q_\ast}

\def\norm#1#2{\|#1\|_{L^#2}}
\def\dfrac#1#2{{\displaystyle\frac{#1}{#2}}}

\usepackage{enumerate}

\newcommand\X{\mathcal{X}}
\newcommand\Y{\mathcal{Y}}

\begin{document}


\title[Aggregation equations ]{Blow-up versus global existence of solutions \\
to  aggregation equations}

\author{Grzegorz Karch}
\address{
 Instytut Matematyczny, Uniwersytet Wroc\l awski,
 pl. Grunwaldzki 2/4, 50-384 Wroc\-\l aw, POLAND}
\email{grzegorz.karch@math.uni.wroc.pl}
\urladdr{http://www.math.uni.wroc.pl/~karch}

\author{Kanako Suzuki}
\address{
Institute for International Advanced Interdisciplinary Research,
Tohoku University,
6-3 Aramaki-aza-Aoba, Aoba-ku, Sendai 980-8578, JAPAN}
\email{kasuzu-is@m.tains.tohoku.ac.jp}

\date{\today}

\thanks{
This work was partially supported %
by the European Commission Marie Curie Host Fellowship for the Transfer of 
Knowledge
``Harmonic Analysis, Nonlinear Analysis and Probability''
MTKD-CT-2004-013389, %
by the Polish Ministry of Science grant N201 022 32/0902 and by %
Japan-Poland Research Cooperative Program (2008-2009).}

\begin{abstract} 
\noindent A class of nonlinear viscous transport equations 
describing aggregation phenomena in biology 
is considered. 
Optimal  conditions on an interaction potential are obtained 
which lead  either to 
 the existence or to the nonexistence of 
global-in-time  solutions. 
\end{abstract}

\bigskip


\maketitle

\section{Introduction}

The following Cauchy problem for the heat equation  corrected by the nonlocal and nonlinear  transport term 
\begin{align}
&u_t = \Delta u - \nabla \cdot \big(u(\nabla K \ast u)\big), \quad x\in \R^n, t>0,\label{eq1} \\
&u(x, 0) =u_0 (x) \label{eq1-ini}
\end{align}
has been used to describe  a collective motion and aggregation phenomena in biology and
mechanics of continuous media.
Here, the unknown function $u = u(x, t)\geq 0$ 
is either the population density of  a species or the density of particles in a granular media.
From a mathematical point of view, equation~\eqref{eq1} can be considered as either a viscous conservation law with a nonlocal (quadratic) nonlinearity or 
a viscous transport equation with nonlocal velocity, and its character depends strongly on the properties of the given kernel $K$.
If this kernel is radially symmetric,
the {\em nonincreasing} function  $K(r)$, $r=|x|$, corresponds to the attraction of particles 
while {\em nondecreasing} one is repulsive.
 
Let us first emphasize that 
 problem \eqref{eq1}--\eqref{eq1-ini} contains, as a particular case, the (simplified) Patlak-Keller-Segel system for chemotaxis 
describing the  motion of cells, usually bacteria or amoebae, that are attracted by a chemical substance and are
able to emit it, see {\it e.g.} \cite{Murray} for a general introduction to chemotaxis.
This parabolic-elliptic system
has  the form
\begin{align}
u_t &= \nabla \cdot (\nabla u - u \nabla v),\quad x\in \R^n, t>0, \label{PKS1} \\
 0&=\Delta v - \alpha v + u,\label{PKS2} 
\end{align}
where $\alpha > 0$ is a given constant. In this model, the function $u=u(x,t)$ represents the cell density and $v=v(x,t)$
is a concentration of the chemical attractant which induces a drift force. 
Computing $v$ from equation \eqref{PKS2}
and substituting it into the transport term in equation \eqref{PKS1}, we immediately obtain equation \eqref{eq1}
with the kernel $K=K(x)$ given by the fundamental solution of the operator $-\Delta +\alpha $ on $\R^n$.
In this case, the function $K$ is called the Bessel potential and it is
singular at the origin if $n \ge 2$, more precisely, it  
satisfies
$
 |\nabla  K (x)| \sim {|x|^{-n+1}} 
$ 
as $|x| \to 0$
and it decays exponentially when $|x|\to\infty$, see \cite[Ch.5.3]{Stein} for
more detail. Hence, when $n \ge 2$, we see that $\nabla  K\in L^{q'}(\R^n)$
for every  $q'<\frac{n}{n-1}$ and $\nabla  K\notin L^{p}(\R^n)$ if $p\geq \frac{n}{n-1}$.
On the other hand, for $n=1$, this fundamental solution is given explicitely by $K(x)=\exp(-\sqrt\alpha |x|)$, hence 
$\nabla  K\in L^{q'}(\R)$ for all $q'\in [1,\infty]$.
We refer the reader to the recent 
 works  \cite{B3,BW,BDP,C-P-Z, K99,Kozono-Sugiyama,Li-Rod1,N2000} (this list is by no mean exhaustive)
and to the references therein for mathematical results on 
 the Patlak-Keller-Segel system \eqref{PKS1}--\eqref{PKS2}.

In this work, we are motivated by recent results on
the
local and global existence of solutions to the {\it inviscid} aggregation equation
\begin{align}
u_t + \nabla \cdot (u (\nabla K \ast u)) = 0, \label{intro-eq02} 
\end{align}
which has been thoroughly studied in \cite{Lau1} under some additional hypotheses on the kernel, see also \cite{BB,BL1}. 
In particular, kernels that are smooth (not  singular) at the origin $x=0$ lead to the global in time existence of solutions, see {\it e.g.} \cite{BL1,BCL1,Lau1}. 
Non-smooth  kernels (and $C^1$ off the origin, like $K(x)=e^{-|x|}$) may lead to blowup of solutions 
either in finite or infinite time \cite{BB,BCL1,BL1,Lau1,Li-Rod1}. 
Singular kernels like potential type (arising in chemotaxis theory, {\it cf. e.g.}~ \cite{BK09, BKL09} and the references therein) 
$ 
K(x)=c|x|^{\beta-d}, 
$
with  $1<\beta<d$ , usually lead to finite time blowup of 
all nonnegative  solutions, see {\it e.g.}~\cite{BKL09}.
 
In particular, in the recent work  by Bertozzi {\it et al.}~\cite{BCL1} on  the inviscid aggregation equation \eqref{intro-eq02},  
the kernel 
$K$  is assumed to be radially symmetric $K(x) = k(|x|)$ with the
function $k(r)$ increasing in $r$, smooth away from zero and bounded from 
below.  The authors of  \cite{BCL1} obtained  natural conditions on $K$ such that all solutions to equation \eqref{intro-eq02},
supplemented with bounded, nonnegative, and compactly
supported initial data,
either
blowup in finite time or exist for all $t > 0$.
More precisely, they introduce the quantity
\begin{align}
 \int_0^1 \frac{1}{k^\prime (r)}\, dr \label{intro-osgood}
\end{align}
and show that 
if \eqref{intro-osgood} is finite, the solution of  \eqref{intro-eq02} blows up in finite
time. On the other hand, if \eqref{intro-osgood} is infinite, 
the global-in-time   solution  to \eqref{intro-eq02} is constructed.

 The purpose of this paper is to describe an analogous influence  of 
singularities  of the kernel $\nabla K$ on the existence and nonexistence of 
global-in-time  solutions of the ``viscous'' problem \eqref{eq1}-\eqref{eq1-ini}. 
 Roughly speaking, our results  can be summarized as follows. 
If $\nabla K \in L^{q'}(\R^n)$ for some $q'\in [1,\infty]$, 
we can always construct local-in-time solutions to \eqref{eq1}--\eqref{eq1-ini}, however, 
some additional regularity assumptions on the initial conditions have to be imposed if $\nabla K$ 
is too singular in the scale of the $L^p$-spaces.
Next, we show  that the initial value problem \eqref{eq1}--\eqref{eq1-ini} with 
a mildly singular interaction  kernel, namely  $\nabla K \in L^{q'}(\R^n)$ for some $q'\in (n,\infty]$,
has a global-in-time solutions for any nonnegative and integrable initial datum \eqref{eq1-ini}. 
On the other hand, 
there are strongly singular kernels, 
such that some solutions of 
problem \eqref{eq1}--\eqref{eq1-ini} blowup in finite time. 
In particular,  we show that  the following behavior
$|\nabla K(x)|\sim {|x|^{-1}}$, as $|x|\to 0$,  appears to be 
 critical for the existence and the nonexistence of global-in-time solutions
to problem \eqref{eq1}--\eqref{eq1-ini}. 
In the next section, we state and discuss our results more precisely.

To conclude this introduction, we would like to mention that completely analogous results 
can be obtained
for the 
aggregation equation 
with the fractional 
dissipation 
\begin{align}
u_t + \nabla \cdot (u (\nabla K \ast u)) = -\nu (-\Delta)^{\gamma /2}u, \label{intro-eq03}
\end{align}
where $\nu > 0$ and $\gamma \in (1, 2]$. 
Some results in this direction, mainly  for the kernel $K$ either of the form  
$K (x) = e^{-|x|}$  or given by the Bessel potential, were  published in 
\cite{BK09, BKL09, Li-Rod1}.


\subsection*{Notation}
Throughout this paper, we denote the norm of the usual Lebesgue space $L^p (\R^n)$, $1 \le p \le +\infty$, by $\| \cdot
\|_{L^p}$. The constants (always independent
of $x$, $t $, and $u$) will be
denoted by the
same letter $C$, even if they may vary from line to line.
Sometimes, we write,  {\it e.g.},  $C=C(T)$ when we want to
emphasize
the dependence of $C$ on a parameter~$T$. We write $f(x)\sim g(x)$ if there is a constant $C>0$ such that
$C^{-1}g(x)\leq f(x)\leq Cg(x)$.


\section{Results and comments}

We begin by introducing terminology systematically used in this work.

\begin{definition}\label{def:ker}
\rm The interaction kernel $K:\R^n\to \R$ is called
\begin{itemize}
\item[i.]
 {\it mildly singular}  if 
$\nabla K \in L^{q'}(\R^n)$ for some $q'\in (n,\infty]$;
\item[ii.]
{\it strongly singular} if 
$\nabla K \in L^{q'}(\R^n)$ for some $q'\in [1,n]$ and $\nabla K \notin L^{p}(\R^n)$ for every $p>n$.
\end{itemize}
\end{definition}

Notice that any function $\nabla K$ satisfying $|\nabla K(x)|\sim |x|^{-a}$ as $|x|\to 0$ and rapidly decreasing if $|x|\to\infty$
is mildly singular in the sense stated above if $a<1$ and strongly
singular for $a\geq 1$. %
Hence,  the Bessel potential  $K$  (appearing in the case of the chemotaxis system)
is strongly singular when $n \ge
2$  and mildly singular for $n=1$.

In order to describe an influence of  singularities  of the function $\nabla K$  on the existence/nonexistence of
solutions to the initial value problem  \eqref{eq1}--\eqref{eq1-ini}, we discuss separately conditions leading to the local-in-time 
existence of solutions, 
their global-in-time existence, as well as the blowup of solutions in finite time .

\subsection*{Local existence of solutions}
First,  we show that the critical exponent $q'=n$ 
from Definition \ref{def:ker}
appears already in the construction of local-in-time solutions to \eqref{eq1}--\eqref{eq1-ini}
with kernels satisfying $\nabla K \in L^{q^\prime} (\R^n)$.
Notice that for strongly singular kernels  we have 
to consider  more regular (in the sense of the $L^p$-spaces) initial
conditions. %
In the following two theorems, the quantity $n/(n-1)$ stands for $+\infty$ if $n =
1$. %
\begin{theorem}[Mildly singular kernels]\label{local-th1}
Assume that $\nabla K \in L^{q^\prime} (\R^n)$ with $ q^\prime \in
 (n,+\infty]$.  
Let $q\in \big[1,\frac{n}{n-1}\big)$ satisfy $1/q + 1/q^\prime = 1$.
For every $u_0 \in L^1 (\R^n)$
 there exists 
$T = T(\|u_0\|_{L^1}, \|\nabla K \|_{L^{q^\prime}}) > 0$ and the
 unique mild solution of problem \eqref{eq1}--\eqref{eq1-ini} in the space
$$
 \X_T= C([0, T],\ L^1 (\R^n)) \cap C((0, T],\ L^q (\R^n))
$$
supplemented with the norm $\|u\|_{\X_T} \equiv \sup_{0 \le t\leq T} \|u\|_{L^1} + \sup_{0 \le
 t\leq T} \left(t^{\frac{n}{2}(1-\frac{1}{q})}\|u\|_{L^q} \right)$.
\end{theorem}

\begin{theorem}[Strongly singular kernels]\label{local-th2}
Assume that $\nabla K \in L^{q^\prime} (\R^n)$ with $ q^\prime \in [1,n]$.
 Let $q \in \big[\frac{n}{n-1},\infty\big]$ satisfy $1/q + 1/q^\prime = 1$.
For every $u_0 \in L^1 (\R^n) \cap L^q
 (\R^n)$, there exists 
$T = T(\|u_0\|_{L^1}, \|u_0\|_{L^q}, \|\nabla K \|_{L^{q^\prime}}) > 0$ and the
 unique mild solution of problem \eqref{eq1}--\eqref{eq1-ini} in the space
\[
 \Y_T= C([0, T],\ L^1 (\R^n)) \cap C([0, T],\ L^q (\R^n))
\]
supplemented with the norm $\|u\|_{\Y_T} \equiv \sup_{0 \le t\leq T} \|u\|_{L^1} + \sup_{0 \le
 t\leq T} \|u\|_{L^q}$.
\end{theorem}

Recall that, as usual, the function $u=u(x,t)$ is called  a {\it mild solution} of \eqref{eq1}--\eqref{eq1-ini} if it satisfies 
the following integral equation
\begin{align}
u(t) = G(\cdot , t)\ast u_0 - \int_0^t \nabla G(\cdot , t-s)* \big(u(\nabla K \ast
 u)\big)(s)\, ds \label{intro-eq1}
\end{align}
with the heat kernel  denoted by 
$ G(x, t) = (4\pi t)^{-n/2}\exp\big(-{|x|^2}/(4t)\big)$. We construct solutions to the integral 
equation \eqref{intro-eq1} using the
Banach contraction principle and, in  the proofs of Theorems \ref{local-th1} and \ref{local-th2}, we emphasize that different estimates are necessary 
according to the singularity of $\nabla K$.

\begin{rem}\label{rem:reg}
In this work, we skip completely questions on regularity of mild solutions to \eqref{eq1}--\eqref{eq1-ini} because there are
standard and well-known results, see {\it e.g.}~the monograph by Pazy \cite{pazy} for more detail.
In particular, by a bootstrap argument, one can show that any mild solution $u\in C([0,T], L^q(\R^n))$ of equation \eqref{intro-eq1}
satisfies
$u\in C^1((0,T], L^q(\R^n))\cap C((0,T], W^{1,q}(\R^n))$ and $u(t)\in W^{2,q}(\R^n)$ for every $t\in (0,T]$.
Moreover, if the initial condition is nonnegative, the same property is shared by the corresponding solution.
\end{rem}

\subsection*{Global existence of solutions}

For  mildly singular kernels, 
 nonnegative solutions to
problem \eqref{eq1}--\eqref{eq1-ini} are global in time.

\begin{theorem}[Mildly singular kernels]\label{global-th1}
Let $q,q'\in [1,\infty]$ satisfy $1/q + 1/q^\prime = 1$.
Assume that $\nabla K \in L^{q^\prime} (\R^n)$ with $ q^\prime \in (n,\infty]$.
 For every $u_0 \in L^1 (\R^n)$ such that 
 $u_0 \ge 0$, there exists the unique global-in-time solution $u$ of problem
 \eqref{eq1}--\eqref{eq1-ini} satisfying
\[
u \in C\big([0, +\infty),\ L^1 (\R^n)\big) \cap C\big((0, +\infty),\ W^{1, q}
 (\R^n)\big) \cap C^1 \big((0, +\infty),\ L^q
 (\R^n)\big).
\]
\end{theorem}

On the other hand, problem \eqref{eq1}--\eqref{eq1-ini} with  strongly singular kernels has a global-in-time solution under 
suitable smallness assumptions imposed on initial conditions. 
To formulate this result, 
it is more convenient to extend the class of considered kernels
and to assume that $\nabla K\in L^{q^\prime,\infty} (\R^n)$, where $ L^{q^\prime,\infty} (\R^n)$ is the weak $L^{q'}$-space
defined as the space of all measurable functions $f$ such that 
$\sup_{\lambda>0} \lambda \big|\{x: |f(x)|>\lambda \}\big|^{1/q'}<\infty.$

Here, let us  recall the well-known embedding $L^{q^\prime} (\R^n)\subset L^{q^\prime,\infty} (\R^n)$ for all $q'\in [1,\infty)$.
However, it follows immediately from the definition of the $ L^{q^\prime,\infty}$-space that
\begin{equation}\label{imbed}
|\cdot|^{-n/q'}\in L^{q^\prime,\infty}(\R^n)\setminus L^{q^\prime}(\R^n)
\quad\mbox{for}\quad 1<q'\leq n.
\end{equation}
In the following, we are going to use the weak Young inequality
\begin{equation}\label{HLS}
\|\nabla K* f\|_{L^k}\leq C\|\nabla K\|_{L^{q',\infty}}\|f\|_{L^p}
\end{equation}
with $p,q',k\in (1,\infty)$ satisfying $1/p+1/q'=1+1/k$, 
a constant $C=C(n,k,p,q')>0$ and  all $f\in L^p(\R^n)$,
see {\it e.g.}~\cite[Sect. 4.3]{LL} for the proof of  \eqref{HLS}.

\begin{theorem}[Strongly singular kernels]\label{global-th2}
Let $n \ge 2$. Assume that $\nabla K\in L^{q^\prime,\infty}(\R^n)$ with $q'\in (1,n]$.
Denote
\begin{equation}
\astq = \frac{n}{n+1-n/q'}\in [1,n).\label{q*}
\end{equation}
There is an $\varepsilon > 0$ such that for every
 $u_0 \in L^{\astq} (\R^n)$ with $\norm{u_0}{{\astq}} < \varepsilon$, %
there exists a global-in-time mild solution of problem \eqref{eq1}--\eqref{eq1-ini} satisfying $u \in C([0, +\infty),\ L^{\astq}(\R^n))$.
\end{theorem}

\begin{rem}\label{rem:Bessel}
Notice that, if $n \ge 2$, the Bessel potential $K= K(x)$ 
satisfies
$\nabla  K\in L^{q^\prime,\infty}(\R^n)$
with $q'=\frac{n}{n-1}$. 
\end{rem}

\subsection*{Blowup versus non-blowup of solutions}

Next, we state conditions on  strongly singular kernels under which 
we can observe the blowup in finite time of solutions to the initial value problem
\eqref{eq1}--\eqref{eq1-ini}.
\begin{theorem}[Strongly singular kernels]\label{blow-th1}
Assume that the kernel $K:\R^n\to \R$ satisfies the following conditions:
\begin{itemize}
\item[i.] $K(x)=K(|x|)$ for all $x \in \R^n$,

\item[ii.] there exist $\delta > 0$, $\gamma>0$, and $\cC > 0$ such that 
\[
 \sup_{0 < s \le \delta}sK^\prime (s) \le
      -\gamma  \quad \text{and}\quad |s K^\prime (s)| \le \cC s^2 \quad \text{for all}\quad
      s \ge \delta. 
\]
\end{itemize}
For the initial datum 
 $u_0 \in
 L^1 (\R^n)$ satisfying  $u_0 \ge 0$ and $|x|^2 u_0 \in L^1
 (\R^n)$, denote
\[
 I(0) = \int_{\R^n} |x|^2 u_0 (x)\, dx \qquad \text{and} \qquad M = \int_{\R^n} u_0(x)\,dx= \norm{u_0}{1}.
\]
If 
\[
 M > \frac{2n + 4(\cC+\gamma/\delta^2) I(0)}{\gamma}, 
\]
then there is
 $T = T(M, I(0), \delta,\gamma,\cC ) > 0$  such that the corresponding nonnegative local-in-time solution to the initial value problem 
\eqref{eq1}--\eqref{eq1-ini} cannot be extended beyond interval $[0,T]$.
\end{theorem}

\begin{rem}\label{rem:p}
Any interaction kernel $K=K(x)$ satisfying the assumptions of~Theorem \ref{blow-th1} has to be 
strongly singular in the sense of Definition \ref{def:ker}.  Indeed, this follows immediately from the following
inequalities 
\begin{equation*}
\begin{split}
\|\nabla K\|_p^p& =\int_{\R^n} \left|  \frac{x}{|x|} K'(|x|) \right|^p\,dx
= \int_{\R^n} \left|   K'(|x|) \right|^p\,dx
= C \int_{0}^\infty  \left|   K'(s) \right|^p  s^{n-1}\,ds\\
&\geq C \gamma^p \int_{0}^\delta  s^{-p +n-1}\,ds=+\infty
\end{split}
\end{equation*}
for every $p\in [n,\infty)$.
Notice that, for   $n=1$,  these calculations  imply that every one 
dimensional kernel $K$ satisfying the assumptions of Theorem \ref{blow-th1}
satisfies also $\nabla K\notin L^p(\R)$ for each $p\in [1,\infty]$.
\end{rem}

\begin{rem}
If $n\geq2$, one can prove, following the reasoning {\it e.g.}~from  \cite[Lem.~3.1]{Kozono-Sugiyama},
that the Bessel potential satisfies the assumptions of Theorem \ref{blow-th1}.
\end{rem}

We conclude the presentation of our results by a non-blowup criterion for 
the initial value problem \eqref{eq1}--\eqref{eq1-ini} with suitable strongly singular kernels.

\begin{theorem}[Strongly singular kernels]\label{strongly-global}
Let $n \ge 1$. 
Assume that $\nabla K \in L^{q^\prime}(\R^n)$ with
 $q^\prime \in [1, n]$  and
let $q \in [\frac{n}{n-1}, \infty]$ satisfy $1/q
 + 1/q^\prime = 1$.
Suppose, moreover,  that the kernel $K$ can be decomposed into two parts,
$
 K = K_1 + K_2,
$
where $\Delta K_1$ is  
nonnegative (as e.g. a tempered distribution)
 and $\nabla K_2
 \in L^\infty (\R^n)$.  Then, for every $u_0 \in L^1 (\R^n) \cap L^q (\R^n)$
 with $u_0 \ge 0$, the local-in-time solution from Theorem \ref{local-th2} exists, in fact, for all $t > 0$.
\end{theorem}

\begin{rem}
Let $n=2$.
The function  $K$ defined by the Fourier transform as
\begin{equation*}
 \widehat{K}(\xi) = \frac{-1}{|\xi|^2 + 1}\label{ss-glob-3}
\end{equation*}
is an example of a strongly singular kernel satisfying the assumptions of Theorem  \ref{strongly-global}
(this is the Bessel potential, discussed just below  \eqref{PKS1}--\eqref{PKS2}, with the reverse sing).
Indeed, using the decomposition
\begin{equation*}
 \widehat{K}(\xi) =\widehat{K_1}(\xi) +\widehat{K_2}(\xi)
\equiv \frac{-1}{|\xi|^2}+ \frac{1}{|\xi|^2(|\xi|^2 + 1)},
\end{equation*}
we see that $\Delta K_1$ is the Dirac delta and $\nabla K_2\in L^\infty(\R^2)$, because $\widehat{\nabla K_2}\in L^1(\R^2)$.
In other words, the two dimensional initial value problem for the parabolic-elliptic system 
\eqref{PKS1}--\eqref{PKS2}, where the sign ``$-$'' in the first equation
is replaced  by ``$+$'', is globally wellposed for any nonnegative initial condition from $L^1(\R^2)\cap L^q(\R^2)$ for some $q\in (2,\infty]$.
\end{rem}

The result from Theorem \ref{strongly-global} on the global-in-time existence of nonnegative solutions is far from 
being optimal. We have stated it here to emphasize the important role of the sing of a strongly singular kernel $K$
in the blowup phenomenon described by Theorem~\ref{blow-th1}.


\section{Construction of local-in-time solutions}

As usual, a solution to the initial value problem \eqref{eq1}--\eqref{eq1-ini} is obtained as a fixed point of the integral equation \eqref{intro-eq1}.
Here, it is  convenient to apply  the following abstract approach proposed by  Meyer \cite{Meyer}.
\begin{lemma}\label{sec1-lem0}
Let $(\X, \| \cdot \|_\X )$ be a Banach space, $y\in\X$,  and $B:\X
 \times \X \rightarrow \X$ be a bilinear form satisfying
$
 \|B(x_1 , x_2)\|_\X \le C \|x_1 \|_\X \|x_2 \|_\X 
$
with a
positive constant $C$ and all $x_1 , x_2 \in \X$.  
If  
$
 4C\|y\|_\X < 1,
$
 the equation
$ x = y + B(x, x)$
has a solution in $\X$ satisfying  $\|x\|_\X \le 2 \|y\|_\X$. Moreover, the
 solution is unique in the ball $U(0, \frac{1}{2C})\subset\X$.
\end{lemma}

We skip the proof of Lemma \ref{sec1-lem0} which is a direct
consequence of the Banach fixed point theorem.

To prove Theorems \ref{local-th1} and
\ref{local-th2},
we are going to apply Lemma \ref{sec1-lem0} 
to the ``quadratic'' equation  \eqref{intro-eq1}, 
written in the form
$
 u( t) = G(\cdot , t)\ast u_0 + B(u, u)(t),
$
with the bilinear form 
\begin{align}
B(u, v) (t) = - \int_0^t \nabla G(\cdot , t-s)* \left(u(\nabla K \ast
 v)\right)(s)\, ds \label{intro-eq2} 
\end{align}
defined on a suitable Banach space. In our reasoning, we use the following well-known estimates 
of the heat kernel which 
are the immediate consequence of the
Young inequality for the convolution:
\begin{align}
\norm{{G(\cdot , t)\ast f}}{p} \le C t^{- \frac{n}{2}\left( \frac{1}{q}-\frac{1}{p} \right)}\norm{f}{q}, \label{intro-eq3}\\
\norm{{\nabla G(\cdot ,t)\ast f}}{p} \le C t^{- \frac{n}{2}
 \left(\frac{1}{q}-\frac{1}{p} \right)- \frac{1}{2}}\norm{f}{q} \label{intro-eq4}
\end{align}
for every $1 \le q \le p \le +\infty$, each $f \in L^q (\R^n)$, and   $C=C(p,q)$ independent of $t,f$.
Notice that $C=1$ in inequality \eqref {intro-eq3} for $p=q$ because $\|G(\cdot,t)\|_{L^1}=1$ for all $t>0$.

\begin{proof}[Proof of Theorem \ref{local-th1}]
First,  we observe  that for every $q^\prime \in (n, +\infty]$, the relation  $1/q + 1/q^\prime = 1$ implies
 $q\in \big[1,\frac{n}{n-1}\big)$.

Here, we use Lemma~\ref{sec1-lem0} with 
the set
$
\X\equiv  \X_T= C([0, T],\ L^1 (\R^n)) \cap C((0, T],\ L^q (\R^n))
$
which is a Banach space 
with the norm $\|u\|_{\X_T} \equiv \sup_{0 \le t\leq T} \|u\|_{L^1} + \sup_{0 <
 t\leq T} \left(t^{\frac{n}{2}(1-\frac{1}{q})}\|u\|_{L^q} \right)$.

By inequality \eqref{intro-eq3},  we immediately obtain
\[
 \|G(\cdot , t)\ast u_0 \|_{L^1} \le \|u_0 \|_{L^1} \quad \text{and}
 \quad t^{\frac{n}{2}(1-\frac{1}{q})} \|G(\cdot , t)\ast u_0 \|_{L^q} 
\le C(q,1) \|u_0 \|_{L^1}
\]
for every  $u_0 \in L^1 (\R^n)$,
hence, $y\equiv G(\cdot , t)\ast u_0 \in \X_T$
with $\|y\|_{\X_T}\leq (1+C(q,1))\|u_0\|_{L^1}$.

Next, we show that 
the bilinear operator defined in \eqref{intro-eq2} satisfies $ B:\X_T
 \times \X_T \rightarrow \X_T$ and there exists a constant $C_1 > 0$ such that for all $T > 0$ and all $u,v\in\X_T$ 
 we have
\begin{equation}
\|B(u, v)\|_{\X_T} \le C_1 T^{\frac{1}{2}\big(1-n(1-\frac{1}{q})\big)} \|\nabla K\|_{L^{q^\prime}} \|u\|_{\X_T}
 \|v\|_{\X_T}.\label{sec1-eq3}
\end{equation}

Assume that $u, v \in \X_T$. First, we compute the $L^1$-norm of $B(u, v)(t)$. By inequalities
 \eqref{intro-eq3} and \eqref{intro-eq4} combined with the 
H\"older inequality and the Young inequality, we have
\begin{align*}
\|B(u, v)(t)\|_{L^1} &\le \int_0^t \|\nabla G(\cdot, t-s) *
 \big(u(\nabla K \ast v)\big)(s)\|_{L^1}\, ds \\
& \le C\int_0^t (t-s)^{-1/2} \|u(\nabla K \ast v)(s)\|_{L^1}\, ds \\
&\le C\int_0^t (t-s)^{-1/2} \|u (s)\|_{L^q} \|\nabla
 K\|_{L^{q^\prime}} \|v(s)\|_{L^1}\, ds \\
&\le C\|\nabla K\|_{L^{q^\prime}} \left(\sup_{0 < s < T}
 s^{\frac{n}{2}(1-\frac{1}{q})}\|u (s)\|_{L^q}\right)\left(\sup_{0 < s
 < T}\|v(s)\|_{L^1}\right)\\
&\quad \times \int_0^t
 (t-s)^{-1/2}s^{-\frac{n}{2}(1-\frac{1}{q})}\, ds \\
&\le C\|\nabla K\|_{L^{q^\prime}} \|u\|_{\X_T}\|v\|_{\X_T} \int_0^t
 (t-s)^{-1/2}s^{-\frac{n}{2}(1-\frac{1}{q})}\, ds,
\end{align*}
where $C$ is a positive constant. Here, notice that
 $-\frac{n}{2}\left(1-\frac{1}{q}\right) > -1$ because $q
 \in[1,\frac{n}{n-1})$, consequently,  
\[
 \int_0^t
 (t-s)^{-1/2}s^{-\frac{n}{2}(1-\frac{1}{q})}\, ds =
 t^{\frac{1}{2}\big(1-n(1-\frac{1}{q})\big)}
 \mathcal{B}\left(1-\frac{n}{2}\left(1-\frac{1}{q}\right),\ \frac{1}{2}\right),
\]
where $\mathcal{B}$ denotes the beta function. Therefore, we obtain
\begin{align}
\sup_{0 < t\leq T}\|B(u, v)\|_{L^1} \le C T^{\frac{1}{2}\big(1-n(1-\frac{1}{q})\big)}\|\nabla K \|_{L^{q^\prime}}
 \|u \|_{\X_T}\|v\|_{\X_T}.\label{sec1-eq1}
\end{align}
where $\frac{1}{2}\big(1-n(1-\frac{1}{q})\big)> 0$.

To deal with the $L^q$-norm of $B(u, v)(t)$, we proceed similarly:
\begin{align*}
t^{\frac{n}{2}(1-\frac{1}{q})}&\|B(u, v) (t)\|_{L^q} \\
&\le
C t^{\frac{n}{2}(1-\frac{1}{q})} \int_0^t (t-s)^{-1/2}\|u (s)\|_{L^q}
 \|\nabla K\|_{L^{q^\prime}} \|v(s)\|_{L^q}\, ds \\
&\le C t^{\frac{n}{2}(1-\frac{1}{q})}\|\nabla K\|_{L^{q^\prime}}
 \left(\sup_{0 \le s < T} s^{\frac{n}{2}(1-\frac{1}{q})} \|u(s)\|_{L^q}
 \right)\left(\sup_{0 \le s < T} s^{\frac{n}{2}(1-\frac{1}{q})}
 \|v(s)\|_{L^q} \right) \\
&\quad \times \int_0^t (t-s)^{-1/2} s^{-n(1-\frac{1}{q})}\, ds \\
&\le C t^{\frac12-\frac{n}{2}(1-\frac{1}{q})}\|\nabla K\|_{L^{q^\prime}}\cB
 \left(1-n\left(1-\frac{1}{q}\right),\ \frac{1}{2}\right)\|u\|_{\X_T} \|v\|_{\X_T}.
\end{align*} 
Hence, we have
\begin{align}
\sup_{0 \le t\leq T} t^{\frac{n}{2}(1-\frac{1}{q})}\|B(u, v) (t)\|_{L^q} \le
C T^{\frac{1}{2}\big(1-n(1-\frac{1}{q})\big)} \|\nabla K\|_{L^{q^\prime}} \|u \|_{\X_T} \|v\|_{\X_T}.\label{sec1-eq2}
\end{align}

Estimates \eqref{sec1-eq1} and \eqref{sec1-eq2} imply  that
 the bilinear form $B$ satisfies \eqref{sec1-eq3}.
Hence, it follows from Lemma \ref{sec1-lem0}  that if we chose $T > 0$ so small that 
\begin{align}
 4C_1 T^{\frac{1}{2}\big(1-n(1-\frac{1}{q})\big)} \|\nabla
 K\|_{L^{q^\prime}}\|u_0 \|_{L^1} (1 + C(q,1))
 < 1, \label{sec1-eq5}
\end{align}
then there exists a solution in the space $\X_T$ with $\|u \|_{\X_T} \le
 2 \norm{u_0}{1} (1 + C(q,1))$. 

By Lemma \ref{sec1-lem0}, this is the unique solution in the ball
 $U(0, \frac{1}{2C})$ with the constant  $C = C_1 T^{\frac{1}{2}\big(1-n(1-\frac{1}{q})\big)} \|\nabla
 K\|_{L^{q^\prime}}$. However, using a standard argument based on the Gronwall lemma combined with the estimates
 leading to \eqref{sec1-eq1} and \eqref{sec1-eq2}, one can show that this is the unique solution in the whole space
$  \X_T$. This completes the proof.
\end{proof}

\begin{proof}[Proof of Theorem \ref{local-th2}]
Now, we assume that $q^\prime
 \in [1, n]$ and we apply Lemma \ref{sec1-lem0}  in the space
 $
 \X\equiv \Y_T= C([0, T],\ L^1 (\R^n)) \cap C([0, T],\ L^q (\R^n))
$
 supplemented with the norm $\|u\|_{\Y_T} \equiv \sup_{0 \le t\leq T} \|u\|_{L^1} + \sup_{0 \le
 t\leq T} \|u\|_{L^q}$. 

By inequality \eqref{intro-eq3}, it is clear that 
\[
\|G(\cdot , t)\ast u_0 \|_{L^1} \le \|u_0 \|_{L^1} \quad \text{and}\quad 
\|G(\cdot , t)\ast u_0 \|_{L^q} \le \|u_0 \|_{L^q}
\]
for every $u_0 \in L^1 (\R^n) \cap L^q (\R^n)$. These inequalities imply that $y\equiv G(\cdot ,
 t)\ast u_0 \in \Y_T$ and $\|G(\cdot ,
 t)\ast u_0 \|_{\Y_T} \le \|u_0 \|_{L^1}+\|u_0 \|_{L^q}$.

Next, for $u, v \in \Y_T$, we see that 
\begin{align*}
\|B(u, v)(t)\|_{L^1} &\le C \int_0^t (t-s)^{-1/2}\|u(s)\|_{L^q} \|\nabla
 K\|_{L^{q^\prime}} \|v(s)\|_{L^1}\, ds \\
&\le C T^{1/2}\|\nabla
 K\|_{L^{q^\prime}} \|u\|_{\Y_T} \|v\|_{\Y_T},
\end{align*}
where $C$ is a positive constant.  In a similar way, we show the following 
$L^q$-estimate
\begin{align*}
\|B(u, v)(t)\|_{L^q} &\le C \int_0^t (t-s)^{-1/2} \|u(\nabla K \ast
 u)(s)\|_{L^q}\, ds \\
&\le C \int_0^t (t-s)^{-1/2}\|u(s)\|_{L^q} \|\nabla
 K\|_{L^{q^\prime}} \|v(s)\|_{L^q}\, ds \\
&\le C T^{1/2}\|\nabla
 K\|_{L^{q^\prime}} \|u\|_{\Y_T} \|v\|_{\Y_T}.
\end{align*}

Summing up these inequalities, we obtain 
\begin{align}
\|B(u, v)\|_{\Y_T} \le C \sqrt{T}\|\nabla
 K\|_{L^{q^\prime}} \|u\|_{\Y_T} \|v\|_{\Y_T}. \label{sec1-eq100}
\end{align}
Therefore, by Lemma \ref{sec1-lem0}, if we
 chose $T > 0$ such that 
\begin{align}
4C \sqrt{T} \|\nabla K \|_{L^{q^\prime}}(\|u_0 \|_{L^1}+\|u_0
 \|_{L^q}) < 1, \label{sec1-eq6}
\end{align}
then we obtain a local-time solution in the space $\Y_T$ which is unique in the open ball
 $U\big(0,  (2C\sqrt{T} \|\nabla K \|_{L^{q^\prime}})^{-1}\big)$.
 However, similarly as in the proof of Theorem \ref{local-th1}, using an argument involving the Gronwall lemma,
 we can show that this is the unique solution in the whole space $\Y_T$.
\end{proof}


\section{Construction of global-in-time solutions}

\begin{proof}[Proof of Theorem \ref{global-th1}]
We are going to show that any nonnegative local-in-time mild solution $u=u(x,t)$ constructed in Theorem
\ref{local-th1} exists, in fact, on every time interval $[0,T]$. 

First, we note that the condition $u_0(x)\geq 0$ implies $u(x,t)\geq 0$ for all $x\in\R^n$ and $t\geq 0$.
Next, integrating equation \eqref{intro-eq1} with respect to $x$, using the Fubini theorem, and the identities
$$
\int_{\R^n}G(x,t)\,dx=1 \quad \mbox{and}\quad \int_{\R^n}\nabla G(x,t)\,dx=0\quad \mbox{for all}\quad t>0,
$$
we obtain the conservation of the $L^1$-norm of nonnegative solutions:
\begin{equation}\label{M:cons}
\|u(t)\|_{L^1}= \int_{\R^n}u(x,t)\,dx=\int_{\R^n}u_0(x)\,dx=\|u_0\|_{L^1}.
\end{equation}
By this reason, the local existence time $T=T(\|u_0\|_{L^1}, \|\nabla K\|_{L^{q'}})$
from Theorem \ref{local-th2} does not change for all nonnegative $u_0\in L^1(\R^n)$ with the same $L^1$-norm.
 For now on,  it suffices to follow a standard procedure which consists in applying repetitiously Theorem~\ref{local-th1}
to equation \eqref{eq1} supplemented with the initial datum $u(x,kT)$ to obtain a unique solution on the interval $[kT, (k+1)T]$
for every $k\in \mathbb{N}$.
This completes the proof of Theorem \ref{global-th1}.
\end{proof}

Next, we deal with strongly singular kernels from the space $L^{q',\infty}(\R^n)$ with $1<q'\leq n$. 
The following lemma plays 
an important role in the proof of Theorem \ref{global-th2}. 
\begin{lemma}\label{sec4-lem1}
Assume that $\nabla K\in L^{q',\infty}(\R^n)$ with $1<q'\leq n$.
For every $r,p\in (1,\infty)$ satisfying
\begin{align}
 \frac{1}{r}= \frac{2}{p} + \frac{1}{q'} - 1. \label{sec4-eq4}
\end{align}
 there is a positive number $C=C(r,p,n,q', \|\nabla K\|_{L^{q',\infty}})$
such that 
for all 
 $u, v \in L^p (\R^n)$  we have
\begin{align}
\|u (\nabla K \ast v)\|_{L^r} \le C\norm{u}{p}\norm{v}{p}.\label{sec4-eq3}
\end{align}
\end{lemma}
\begin{proof}
First, one should use the H\"older inequality to estimate
$$
\|u (\nabla K \ast v)\|_{L^r} \le C\norm{u}{p}
\norm{\nabla K*v}{k}\quad \mbox{with} \quad  \frac{1}{r}=\frac{1}{p}+\frac{1}{k}.
$$
Next, we apply the weak Young  inequality \eqref{HLS} which
 leads to $\nabla K \ast v \in L^k (\R^n)$ with
$
 {1}/{k} = {1}/{p} + {1}/{q'} - 1.
$
\end{proof}

\begin{proof}[Proof of Theorem \ref{global-th2}]
Recall that $q_*=1/(1+1/n-1/q')$.
For an exponent $p$ satisfying
\begin{align}
1 \le  
\max\left\{q_*,
\frac{1}{1-1/(2q')}\right\}  < p <  \frac{1}{1-1/q'+1/(2n)} \label{sec4-eq1},
\end{align}
we define the Banach space
\[
 \X = C([0, +\infty), L^{\astq}(\R^n)) \cap \left\{C([0, +\infty), L^p
 (\R^n)) \mid \sup_{t > 0}
 t^{\frac{n}{2}\left(\frac{1}{\astq} -
 \frac{1}{p}\right)}\norm{{u(t)}}{p} < +\infty \right\}
\]
with the norm $\|u\|_{\X} \equiv \sup_{t > 0}\norm{{u(t)}}{{\astq}} + \sup_{t > 0}
 t^{\frac{n}{2}\left(\frac{1}{\astq} -
 \frac{1}{p}\right)}\norm{{u(t)}}{p}$. %

For every $u_0 \in L^{{\astq}} (\R^n)$, it follows immediately from
 estimates \eqref{intro-eq3} that 
\begin{align}
 \|G(\cdot)*u_0 \|_{\X} \le C_3 \norm{{u_0}}{{\astq}} \label{sec4-eq100}
\end{align}
for some constant $C_3 > 0$. 

In the next step, we estimate the bilinear form $B(u,v)$ defined in \eqref{intro-eq2}
for any $u, v \in \X$.
By estimates \eqref{intro-eq4} and \eqref{sec4-eq3}, we have
\begin{align}
\| B(u, v)(t)\|_{L^{\astq}} &\le C \int_0^t
 (t-s)^{-\frac{n}{2}\left(\frac{1}{r}-\frac{1}{\astq}\right)-\frac{1}{2}}
 \|u(\nabla K \ast v)(s)\|_{L^r} \, ds \nonumber\\
&\le C \int_0^t
 (t-s)^{-\frac{n}{2}\left(\frac{1}{r}-\frac{1}{\astq}\right)-\frac{1}{2}}
 \|u(s)\|_{L^p} \|v(s) \|_{L^p} \, ds \label{sec4-eq5}\\
&\le C \|u\|_{\X} \|v\|_{\X} \int_0^t
 (t-s)^{-\frac{n}{2}\left(\frac{1}{r}-\frac{1}{\astq}\right)-\frac{1}{2}}
 s^{-n \left(\frac{1}{\astq} - \frac{1}{p}\right)}\, ds, \nonumber
\end{align}
where $r$ is defined in \eqref{sec4-eq4}. 
Inequalities in \eqref{sec4-eq5} make sense and involve convergent integrals because, by a direct calculation, 
it follows from
 \eqref{sec4-eq1} that 
\[
1<r\leq q^*, \quad 
 -\frac{n}{2}\left(\frac{1}{r} - \frac{1}{\astq}\right)-\frac{1}{2} > -1,
 \quad \text{and}\quad -n\left(\frac{1}{\astq} - \frac{1}{p}\right) > -1.
\]
Therefore, after changing the variable on the right-hand side of
 \eqref{sec4-eq5}, we see that 
\begin{align*}
&\| B(u, v)(t)\|_{L^{\astq}} \\
&\le C \|u\|_{\X} \|v\|_{\X}
 t^{-\frac{n}{2}\left(\frac{1}{r}-\frac{1}{\astq}\right)-\frac{1}{2} -
 n\left(\frac{1}{\astq} - \frac{1}{p}\right) + 1}\cB
 \left(1-n\left(\frac{1}{\astq} - \frac{1}{p}\right),
 \frac{1}{2}\left(1-n\left(\frac{1}{r}- \frac{1}{\astq}\right)\right)\right),
\end{align*}
where $\cB$ denotes the beta function.
However, for $q^*$ defined by \eqref{q*}, it follows from relation \eqref{sec4-eq4} that %
\[
 -\frac{n}{2}\left(\frac{1}{r}-\frac{1}{\astq}\right)-\frac{1}{2} -
 n\left(\frac{1}{\astq} - \frac{1}{p}\right)+ 1 =0,
\]
hence,
the $L^{\astq}$-norm of $B(u,v)$ is estimated as
\begin{align}
\sup_{t > 0}\| B(u, v)(t)\|_{L^{\astq}} \le C \|u\|_{\X} \|v\|_{\X} \label{sec4-eq6}
\end{align}
with a  positive constant $C$.

By  similar arguments as those in the case of the  $L^{\astq}$-estimate, we obtain
\begin{align*}
&t^{\frac{n}{2}\left(\frac{1}{\astq}-\frac{1}{p}\right)}\|B(u,
 v)(t)\|_{L^p} \nonumber\\
&\le C \|u\|_{\X} \|v\|_{\X} t^{\frac{n}{2}\left(\frac{1}{\astq}-\frac{1}{p}\right)}\int_0^t
 (t-s)^{-\frac{n}{2}\left(\frac{1}{r}-\frac{1}{p}\right)-\frac{1}{2}}
 s^{-n \left(\frac{1}{\astq} - \frac{1}{p}\right)}\, ds \nonumber \\
&= C \|u\|_{\X} \|v\|_{\X}.
\end{align*} 
Therefore, there is a constant $C > 0$ independent of $t$ such that 
\begin{align}
\sup_{t > 0} t^{\frac{n}{2}\left(\frac{1}{\astq}-\frac{1}{p}\right)}\|B(u,
 v)(t)\|_{L^p} \le C \|u\|_{\X} \|v\|_{\X} . \label{sec4-eq7}
\end{align}

Finally, it follows from \eqref{sec4-eq6} and \eqref{sec4-eq7} that 
\begin{align}
 \|B(u, v)\|_{\X} \le \eta \|u\|_{\X} \|v\|_{\X} \label{sec4-eq8}
\end{align}
for a positive number $\eta$ independent of $t$, $u$, and $v$.
Hence,
we conclude  by Lemma~\ref{sec1-lem0} that  the equation 
$
 u(t) = G(t)* u_0 + B(u, u)
$ 
has a solution in $\X$ if 
$
 4\eta\|G(t)*u_0 \|_{\X} < 1.
$
However, by
\eqref{sec4-eq100}, it suffices to assume that  
$
 \|u_0 \|_{L^{\astq}} < \frac{1}{4\eta C_3}
$
to complete the proof of Theorem~\ref{global-th2}.
\end{proof}


\section{Nonexistence of global-in-time solutions}

\begin{proof}[Proof of Theorem \ref{blow-th1}]
Let us recall that we limit ourselves to nonnegative solutions to \eqref{eq1}--\eqref{eq1-ini} satisfying 
$$
M = \int_{\R^n} u(x, t)\, dx = \int_{\R^n} u_0 (x)\, dx \quad\mbox{for all} \quad t\in [0,T].
$$ 
As a standard practice, we
 study the evolution of the second moment of a solution to \eqref{eq1}--\eqref{eq1-ini}
\[
 I(t) = \int_{\R^n} |x|^2 u (x, t)\, dx.
\]
Here, we skip the well-known argument (see~{\it e.g.}~\cite{Kozono-Sugiyama}) saying that the quantity $I(t)$ is
finite if $u_0\in L^1\big(\R^n, (1+|x|^2)\,dx\big)$.

Differentiating the function $I(t)$ with respect to $t$,
using equation \eqref{eq1}, and integrating by parts,
 we obtain
\begin{align}
\frac{d}{dt}I(t) &= \int_{\R^n} |x|^2 \left(\Delta u - \nabla \cdot (u
 (\nabla K \ast u))\right)\, dx \nonumber\\
& = 2n M + 2 \int_{\R^n} x \cdot u (\nabla K \ast u)\, dx \nonumber\\
&= 2nM + 2 \int_{\R^n} \int_{\R^n} u (x, t) u(y, t) x \cdot \nabla
 K(x-y)\, dx dy. \label{sec5-eq100}
\end{align}
Symmetrizing in $x$ and $y$ the double integral on the right-hand side of
 \eqref{sec5-eq100}, we obtain
\begin{align}
\frac{d}{dt}I(t) = 2nM + \int_{\R^n} \int_{\R^n} u (x, t) u(y, t) \big(x \cdot \nabla
 K(x-y) + y \cdot \nabla K(y-x) \big) \, dx dy. 
\end{align}
Now, notice that the interaction kernel is assumed to be radial, $K(x) = K(|x|)$, hence $\nabla K (x) =
 \dfrac{x}{|x|}K^\prime (r)$, where $r = |x|$. Therefore, we see that 
\begin{align*}
x \cdot \nabla K(x-y) + y \cdot \nabla K(y-x) &= x \cdot
 \frac{x-y}{|x-y|} K^\prime (|x-y|) + y \cdot \frac{y-x}{|y-x|} K^\prime (|y-x|)\\
&= |x-y| K^\prime (|x-y|).
\end{align*}
Now, we apply the assumption ii. imposed on the kernel $K$ as follows
\begin{align*}
\frac{d}{dt} I(t) &= 2nM + \int_{\R^n} \int_{\R^n} u (x, t) u(y, t) |x-y| K^\prime (|x-y|)\,
 dx dy \\
 &\le 2nM - \gamma
  \int\!\!\!\int_{|x-y| \le \delta} u (x, t) u(y, t)\,dx dy + \\
&\quad + \cC \int\!\!\!\int_{|x-y| > \delta} u (x, t) u(y, t) |x-y|^2 \,dx dy \\
&\le 2nM - \gamma  M^2 + (\cC+\gamma/\delta^2)  \int\!\!\!\int_{|x-y| > \delta} u (x, t) u(y, t) |x-y|^2 \,dx dy.
\end{align*}
Hence, using the elementary inequality $|x-y| \le 2\big(|x|^2 + |y|^2\big)$ we obtain
\begin{equation*}\label{sec5-eq1}
\begin{split}
\frac{d}{dt}I(t) &\le 2nM - \gamma  M^2 + 2 (\cC+\gamma/\delta^2) \int_{\R^n}\int_{\R^n} u (x, t) u(y, t) (|x|^2 +
 |y|^2) \,dx dy \\
&= M \big( 2n 
-\gamma  M 
+ 4(\cC+\gamma/\delta^2)
 I(t)
\big), 
\end{split}
\end{equation*}
which implies that 
\begin{equation*}
\frac{d}{dt}I(t)\leq 
M \big( 2n 
-\gamma  M 
+ 4(\cC+\gamma/\delta^2)
 I(0)
\big)
< 0 \quad \mbox{for all}\quad  t > 0 
\end{equation*} provided
$
\gamma M > 2n + 4(\cC+\gamma/\delta^2) I(0). 
$
Consequently, $I(T)=0$ for some $0<T<\infty$. 
This contradicts the global-in-time existence of regular nonnegative solutions of problem \eqref{eq1}--\eqref{eq1-ini}.
\end{proof}

\begin{proof}[Proof of Theorem \ref{strongly-global}]
In order to show that the local-in-time solution from Theorem~\ref{local-th2} exists for all $t\in [0,\infty)$,
it is sufficient to obtain its {\it a priori} $L^q$-estimate. 
Indeed, if $\|u(t)\|_{L^q}$ does not blow up in finite time, we can apply a continuation argument analogously as in 
the proof of Theorem \ref{global-th1}.

Multiplying both sides of equation \eqref{eq1}  by $u^{q-1}$  (recall that $u$ is nonnegative),
integrating  over
 $\R^n$, and using the decomposition of $K$, we have
\begin{align}
\frac{1}{q}\frac{d}{dt}\int_{\R^n} u^q \, dx %
&= \int_{\R^n} u^{q-1}\Delta u\, dx - \int_{\R^n} u^{q-1}\nabla \cdot \left(u
 (\nabla K \ast u )\right)\, dx \nonumber\\
&= -(q-1)\int_{\R^n} u^{q-2}|\nabla u |^2 \, dx + (q-1)\int_{\R^n}
  u^{q-1}\nabla u \cdot (\nabla K_1 \ast u)\, dx \label{ss-glob-1} \\
&\quad + (q-1)\int_{\R^n}
 u^{q-1} \nabla u \cdot (\nabla K_2 \ast u)\, dx.  \nonumber
\end{align}
Notice that
the second term of the right-hand side of \eqref{ss-glob-1} is nonpositive 
 due to the assumptions of $K_1$ 
in  view of the following calculation
\begin{align*}
(q-1)\int_{\R^n}
 u^{q-1} \nabla u \cdot (\nabla K_1 \ast u)\, dx &= \frac{q-1}{q}\int_{\R^n}
 \nabla u^{q} \cdot (\nabla K_1 \ast u)\, dx \\
&= -\frac{q-1}{q}\int_{\R^n}
 u^{q} \cdot (\Delta K_1 \ast u)\, dx
\le 0.
\end{align*}
Here, we have assumed $K_1$ to be sufficiently regular and the more general case can be handled by a standard 
regularization procedure.

Next, 
by the $\varepsilon$-Young inequality, the third term of the right-hand side of \eqref{ss-glob-1} is estimated
 as follows
\begin{align*}
(q-1)\int_{\R^n}&
 u^{q-1} \nabla u \cdot (\nabla K_2 \ast u)\, dx \\
&\ \le
 (q-1)\left[\varepsilon \int_{\R^n} u^{q-2}|\nabla u|^2 \, dx +
 C(\varepsilon) \int_{\R^n} u^q |\nabla K_2 \ast u |^2 \, dx \right]\\
&\ \le \varepsilon (q-1) \int_{\R^n} u^{q-2} |\nabla u |^2 \, dx +
 C(\varepsilon)\|\nabla K_2 \|_{L^\infty}^2 \|u _0\|_{L^1}^2 \int_{\R^n} u^q \, dx
\end{align*}
since, by \eqref{M:cons}, we have 
$
\|\nabla K_2 * u (t)\|_{L^\infty}\leq
\|\nabla K_2 \|_{L^\infty} \|u _0\|_{L^1}
$.

Therefore, coming back to \eqref{ss-glob-1}, for $\varepsilon\leq 1$,  we see that 
\begin{align*}
\frac{1}{q}\frac{d}{dt} \int_{\R^n} u^q \, dx &\le
 -(q-1)(1-\varepsilon)\int_{\R^n} u^{q-2}|\nabla u|^2 \, dx +
 C(\varepsilon)\|\nabla K_2 \|_{L^\infty}^2 \|u _0\|_{L^1}^2 \int_{\R^n} u^q \,
 dx \\
&\le C(\varepsilon)\|\nabla K_2 \|_{L^\infty}^2 \|u_0 \|_{L^1}^2 \int_{\R^n} u^q \, dx.
\end{align*}
Hence, by the Gronwall lemma, 
$
 \|u(t)\|_{L^q} \le e^{C t}\|u_0 \|_{L^q},
$
where $C =  C(\varepsilon)\|\nabla K_2 \|_{L^\infty}^2 \|u_0 \|_{L^1}^2$. This
 implies that $\|u (t)\|_{L^q}$ does not blow up in finite time and the proof of Theorem~\ref{strongly-global}
is complete. 
\end{proof}


%

\end{document}